\documentclass[sn-basic]{sn-jnl}
\theoremstyle{thmstyleone}%
\newtheorem{theorem}{Theorem}
\newtheorem{corollary}[theorem]{Corollary}
\newtheorem{proposition}[theorem]{Proposition}
\newtheorem{lemma}[theorem]{Lemma}%
\newcommand\sauri[1]{\textcolor{black}{#1}}
\theoremstyle{thmstylethree}%
\newtheorem{remark}{Remark}%

\theoremstyle{thmstylethree}%
\newtheorem*{assumption}{Assumption} 
\begin{document}

\title[Local Limit Theorems for Energy Fluxes of Infinite Divisible Random Fields]{Local Limit Theorems for Energy Fluxes of Infinite Divisible Random Fields}


\author[1,2]{\fnm{Jos\'e Ulises} \sur{M\'arquez-Urbina}}\email{ulises@cimat.mx}

\author[3]{\fnm{Orimar} \sur{Sauri}}\email{osauri@math.aau.dk}

\affil[1]{\orgdiv{Centro de Investigaci\'on en Matem\'aticas }, \orgname{Unidad Monterrey}, \orgaddress{\street{Av. Alianza Centro No. 502, PIIT}, \city{Apodaca}, \postcode{66628}, \state{N.L.}, \country{Mexico}}}

\affil[2]{\orgname{Consejo Nacional de Ciencia y Tecnolog\'ia}, \orgaddress{\street{Av. Insurgentes Sur 1582}, \city{CDMX}, \postcode{03940}, \country{Mexico}}}

\affil[3]{\orgdiv{Department of Mathematical Sciences}, \orgname{Aalborg University}, \orgaddress{\street{ Skjernvej 4, A, 5-202}, \city{Aalborg \O{}.}, \postcode{9220}, \country{Denmark}}}


\abstract{We study the local asymptotic behavior of divergence-like functionals of a family of $d$-dimensional Infinitely Divisible Random Fields. Specifically, we derive limit theorems of surface integrals over Lipschitz manifolds for this class of fields when the region of integration shrinks to a single point.  We show that in most cases, convergence stably in distribution holds after a proper normalization. Furthermore, the limit random fields can be described in terms of stochastic integrals with respect to a L\'evy basis. We additionally discuss how our results can be used to measure the kinetic energy of a possibly turbulent  flow.}



\pacs[Acknowledgment]{ We dedicate this article to Ole Barndorff-Nielsen's memory (1935-2022). Thanks for all your support.}
\maketitle\
\section{Introduction\label{intro}}
\subsection{ Overview}Kinetic energy is the energy associated with a body due to its motion. Formally, the kinetic energy of a body is defined as the work needed to accelerate it from rest to its stated velocity. In this work, we are interested in the local behaviour of the kinetic energy of a turbulent flow.  In physics, turbulence refers to the chaotic and unpredictable motions found in some fluids which is typically characterised by abrupt changes in pressure and flow velocity. In a turbulent flow, the kinetic energy flux measures the amount of energy being injected or extracted from the fluid enclosed in a region. Thus, it is a proxy of energy dissipation in a turbulent flow. 

Understanding turbulence is considered one of the last open problems of classical physics. As part of fluid dynamics, turbulence can be studied via the Navier-Stokes equations. However, this approach has proven to be very challenging; therefore, numerous efforts have been made to develop phenomenological models that reproduce some of the key stylized features of turbulent fluids. Such models aim to produce tools that can be employed in practical situations or allow the understanding of some turbulence elements. Ambit processes stand out among these phenomenological models due to their flexibility and theoretical properties. These stochastic processes were introduced as models for turbulent velocity flows in \cite{BNSch03}; they provide a robust framework to describe spatio-temporal phenomena and have been applied in different contexts like finance (\cite	{BNBEnthVeraat14,BNBEnthVeraat13}), tumor growth (\cite{BNSch07,BNJenJonsdSchm07}), and turbulence (\cite{BNSch07,BNSch03,HedSch13,HedSch14}). In broad terms, ambit processes are a general class of spatio-temporal stochastic processes defined as stochastic integrals with respect to an independently scattered and infinitely divisible random measure. We refer the reader to \cite{BNBEnthVeraat18} for more details on ambit stochastics.

This article studies local limits of general energy fluxes over smooth manifolds for a subclass of vector-valued ambit fields. Besides the purely mathematical interest, studying (kinetic) energy fluxes of random fields could shed some light on the conditions the chosen model requires to fulfill in order to reproduce key features present in turbulent flows.

\subsection{ Related work}From the perspective of modeling turbulence, there is some literature related to the present work. \cite{BNSch03} introduced the class of ambit processes and proposed employing them to model the energy dissipation of a turbulent flow. \cite{BNSch07} discussed for the first time the use of ambit processes to model a turbulent velocity field; in that article, the authors also discussed some relevant questions required to aim for a complete theory of ambit processes for turbulence. \cite{HedSch14} proposed specific ambit random fields capable of reproducing a given covariance structure. In particular, it was shown that in the isotropic and incompressible case, the kernel is expressible in terms of the energy spectrum; the models developed are applied to atmospheric boundary layer turbulence. \cite{Sch20} discusses the use of ambit random fields for the description of 2-dimensional turbulence. In that work, the author discusses the construction of 2-dimensional homogeneous and isotropic ambit fields that are divergence-free but not invariant under the parity operation. 

On the other hand, to the best of our knowledge, the questions addressed in the present work have only been previously considered in two manuscripts. The first one is the work in \cite{Sch20} discussed before. In this set-up, one can describe energy fluxes via classical vector calculus. In a more broad framework,  \cite{Sauri20} studies the flux and circulation of a 2-dimensional subclass of ambit random fields, determining local limits for those functionals under proper normalization. Namely, it is shown that in most cases they converge stably in distribution towards stationary random fields given in terms of line integrals of a Levy basis over the boundary of the original underlying ambit set. Other mathematical works that have considered similar functionals to the one studied in this article can be found in the theory of statistical mechanics and microstructures in continuum mechanics (e.g., \cite{presutti2008scaling}); in that area, for example, the macroscopic excess free energy is defined as a surface integral. Although some limiting behavior is addressed in this theory, they do not study the limits of functionals of random fields defined by integrals with respect to the Haussdorff measure as we do in the present work.

\subsection{  Main contributions of this article} We study the asymptotic behavior in divergence-like limits for fluxes of infinitely divisible random fields of the form
\begin{align*}
	X(p)=\int_{A+p} F(p,q) L(dq) \quad p\in\mathbb{R}^d,
\end{align*}
\sauri{where $L$ is a homogeneous Lévy basis (see the next section for more details), $F$ is continuously differentiable} and $A$ compact. More precisely, we determine conditions for the convergence, as $r\downarrow 0$, of  normalised functionals of the form 
\begin{align*}
	\mathscr{E}_{r}=\int_{S_r} \phi (X(y))\cdot u_{S_r}\mathcal{H}^{d-1}(dy), 
\end{align*}
where $S_r=rM+p_0$ is the boundary of a region $V_r=r\mathfrak{D}+p_0$, $\phi$ is a function with polynomial growth, $u_{S_r}$ is the outward unit normal to $S_r$, and $\mathcal{H}^{d-1}$ denotes the $(d-1)$-Hausdorff measure in $\mathbb{R}^d$. It turns out that the rate of convergence of $\mathscr{E}_{r}$ strongly depends on whether $L$ is of finite variation or not. In the latter situation, our central assumption is that the law of the ``small jumps'' of $L$ belongs to the domain of attraction of an $\alpha$-stable distribution. In the  finite variation case, we further show that the kinetic energy flux converges in probability under the classical normalization $\lvert S_r\lvert $.
In both situations, the limit processes can be expressed in terms of stochastic integrals with respect to a L\'evy basis over regions uniquely determined by the geometry of $A$. Finally, by considering  $(\mathscr{E}_{tr})_{t\geq0}$ as a sequence of continuous-time stochastic process, we show that the limiting process of such sequence is not only self-similar but also absolutely continuous, regardless of whether $L$ is of finite variation or not.

The organization of the paper is as follows. In Section 1, we introduce the basic probabilistic and geometrical concepts and results that will be used in our work. Section 3 describes our main results regarding the asymptotic behaviour of energy fluxes and related functionals. Due to the technical nature of our proofs, most of them will be presented in Section 4.

\section{Preliminaries\label{sec:Preliminaries-and-basic}}

This part is devoted to introducing the basic notations as well as to recall several basic results and concepts that will be used through
this paper.

\subsection{Stable convergence and L\'evy bases\label{Preliminariesstableconv}}

In this work, the inner product and the norm of vectors $x,y\in\mathbb{R}^{d}$
will be represented by $x\cdot y$ and $\| x\| $,
respectively. Throughout the following sections $\left(\Omega,\mathcal{F},\mathbb{P}\right)$
will denote a complete probability space. For a sequence of random vectors $(\xi_{n})_{n\geq1}$ defined on $\left(\Omega,\mathcal{F},\mathbb{P}\right)$, we write $\xi_{n}=\mathrm{o}_{\mathbb{P}}(1)$ whenever $\xi_{n}\overset{\mathbb{P}}{\rightarrow}0$, as $n\rightarrow\infty$. Furthermore, given a sub-$\sigma$-field $\mathcal{G}\subseteq\mathcal{F}$ and a random vector $\xi$ (defined possibly on an extension of $\left(\Omega,\mathcal{F},\mathbb{P}\right)$), we say that $\xi_{n}$ converges $\mathcal{G}$-stably in distribution towards $\xi$, and write $\xi_{n}\overset{\mathcal{G}\text{-}d}{\longrightarrow}\xi$, if for any $\mathcal{G}$-measurable random variable (r.v. from now on) $\zeta$, $(\xi_{n},\zeta)\rightarrow(\xi,\zeta)$ weakly as $n\rightarrow\infty$. Within the preceding framework, if $(X_{n}(t))_{t\in T,n\in\mathbb{N}}$ is a sequence of random fields defined on $\left(\Omega,\mathcal{F},\mathbb{P}\right)$,
we will write $X_{n}\overset{\mathcal{G}\text{-}fd}{\longrightarrow}X$ if the finite-dimensional distributions (f.d.d. for short) of $X_{n}$
converge $\mathcal{G}$-stably toward the f.d.d. of $X$. For a concise exposition of stable convergence, see \cite{HauslerLuschgy15} and references therein.

Let $\mu$ be a measure on $\mathcal{B}(\mathbb{R}^{d})$, the Borel sets on $\mathbb{R}^{d}$, and define $\mathcal{B}_{b}^{\mu}(\mathbb{R}^{d}):=\{A\in\mathcal{B}(\mathbb{R}^{d}):\mu(A)<\infty\}$. The $\mathbb{R}^{m}$-valued random field $L=\{L\left(A\right):A\in\mathcal{B}_{b}^{\mu}(\mathbb{R}^{d})\}$ will be called a \textit{separable L\'evy basis} with \textit{control measure} $\mu$ if it satisfies the following:
\begin{enumerate}
	\item For every $A\in\mathcal{B}_{b}^{\mu}(\mathbb{R}^{d})$, $L(A)$ is infinitely divisible (ID for short). 
	\item $L(A)$ and $L(B)$ are independent whenever $A,B\in\mathcal{B}_{b}^{\mu}(\mathbb{R}^{d})$ and $A\cap B=\emptyset$. 
	\item Given a disjoint sequence $\left\lbrace A_{n}\right\rbrace _{n\geq1}\subseteq\mathcal{B}_{b}^{\mu}(\mathbb{R}^{d})$ such that $\cup_{n=1}^{\infty}A_{n}\in\mathcal{B}_{b}^{\mu}(\mathbb{R}^{d})$,
	it holds almost surely (a.s. for short) 
	\[      
	L(\cup_{n=1}^{\infty}A_{n})=\sum_{n\geq1}L(A_{n}).
	\]
	\item For every $A\in\mathcal{B}_{b}^{\mu}(\mathbb{R}^{d})$ and $z\in\mathbb{R}^{m}$, we have that 
	\[
	\mathbb{E}(\exp(\mathbf{i}z\cdot L(A)))=\exp(\mu(A)\psi(z)),
	\]
	where
	\begin{equation}
		\psi(z):=\mathbf{i}\gamma\cdot z-\frac{1}{2}z\cdot\Sigma z+\int_{\mathbb{R}^{m}\backslash\{0\}}(e^{\mathbf{i}z\cdot x}-1-\mathbf{i}z\cdot x\mathbf{1}_{\| x\| \leq1})\nu(\mathrm{d}x), \label{chexpdef}
	\end{equation}
	with $\gamma\in\mathbb{R}^{m},$ $\Sigma$ a $m\times m$ positive
	definite matrix and $\nu$ a L\'evy measure, i.e. $\nu(\{0\})=0$
	and $\int_{\mathbb{R}^{m}\backslash\{0\}}(1\land\| x\| ^{2})\nu(\mathrm{d}x)<\infty$.
\end{enumerate}
When $\mu=\mathrm{Leb}$, in which $\mathrm{Leb}$ represents the Lebesgue measure on $\mathbb{R}^{d}$, $L$ is called \textit{homogeneous}. The ID random vector associated with the characteristic triplet $\left(\gamma,\Sigma,\nu\right)$ is known as the \textit{L\'evy seed} of $L$, and it will be denoted by $L'$. As usual, $\left(\gamma,\Sigma,\nu\right)$ will be called the characteristic triplet of $L$ and $\psi$ its characteristic exponent.

Any non-zero L\'evy measure on $\mathbb{R}^{m}$ admits a \textit{polar decomposition} 
\begin{equation}
	\nu(B)=\int_{\mathbb{S}^{m-1}}\int_{0}^{\infty}\mathbf{1}_{B}(ru)\rho_{u}(\mathrm{d}r)\lambda(\mathrm{d}u),\label{eq:polardecomp}
\end{equation}
where $\mathbb{S}^{m-1}$ is the unitary sphere in $\mathbb{R}^{m}$, $\lambda$ is a finite measure on $\mathbb{S}^{m-1}$, and $\{\rho_{u}:u\in\mathbb{S}^{m-1}\}$ is a family of L\'evy measures on $(0,\infty)$ such that the mapping $u\mapsto\rho_{u}(B)$ is measurable for all $A\in\mathcal{B}((0,\infty))$.

Let $0<\alpha\leq2$ and $\lambda$ a finite measure on $\mathbb{S}^{m-1}$. A separable L\'evy basis is called \textit{strictly $\alpha$-stable} if its L\'evy seed is distributed according to a strictly $\alpha$-stable distribution with spectral measure $\lambda$; that is, $L'$ is centred Gaussian with covariance $\Sigma$ if $\alpha=2$, while for $0<\alpha<2$ the characteristic triplet of $L'$ has no Gaussian component ($\Sigma=0$), its L\'evy measure admits the polar decomposition 
\[
\nu(B)=\int_{\mathbb{S}^{m-1}}\int_{0}^{\infty}\mathbf{1}_{B}(ru)\frac{\mathrm{d}r}{r^{1+\alpha}}\lambda(\mathrm{d}u),
\]
and $\gamma=\frac{\int_{\mathbb{S}^{m-1}}u\lambda(\mathrm{d}u)}{(1-\alpha)}$ if $\alpha\neq1$, while if $\alpha=1$, $\gamma$ can be arbitrary but with the restriction that $\int_{\mathbb{S}^{m-1}}u\lambda(\mathrm{d}u)=0$. For $\alpha<2$, the characteristic exponent of a strictly $\alpha$-stable L\'evy basis can be written as
\begin{equation}
	\psi_{\alpha}(z):=
	\begin{cases}
		-\int_{\mathbb{S}^{m-1}}\mid z\cdot u\mid^{\alpha}\varphi_{\alpha}(z,u)\lambda(\mathrm{d}u) & \text{if }\alpha\neq1;\\
		-\int_{\mathbb{S}^{m-1}}\mid z\cdot u\mid\varphi_{\alpha}(z,u)\lambda(\mathrm{d}u)+\mathbf{i}\gamma z & \text{if }\alpha=1,
	\end{cases}
	\label{chfnctstrictlystabe}
\end{equation}
where 
\[ 
\varphi_{\alpha}(z,u)=
\begin{cases}
	1-\mathbf{i}\rho\mathrm{sign}(z\cdot u)\tan(\pi\beta/2) & \text{if }\alpha\neq1;\\
	1+\mathbf{i}\frac{2}{\pi}\mathrm{sign}(z\cdot u)\log(\mid z\cdot u\mid) & \text{if }\alpha=1.
\end{cases} 
\]
For the facts and concepts discussed in this section, we refer the reader to \cite{Sato99,SamorodTaqqu94,Rosinski90}.

\subsection{Geometrical preliminaries\label{PreliminariesJCHausdorff}}

For any $A\subseteq\mathbb{R}^{d},$ we let $-A=\{-x:x\in A\}$. Furthermore, we denote by $\mathring{A},\bar{A},\partial A,$ and $A^{c}$ the
interior, the closure, the boundary, and the complement of $A$, respectively, and we put $A^{*}=\bar{A^{c}}$. An open set $\mathfrak{D}\subseteq\mathbb{R}^{d}$ is said to be a \textit{Lipschitz domain} if its boundary can be locally described as the graph of a Lipschitz function defined on an open set of $\mathbb{R}^{d-1}$. We will say that a $(d-1)$-dimensional manifold $M\subseteq\mathbb{R}^{d}$ is Lipschitz if it is the boundary of a Lipschitz domain. For $s>0$, the $s$-dimensional Hausdorff measure will be represented by $\mathcal{H}^{s}$. Now, fix $A\subseteq\mathbb{R}^{d}$ a closed set. The \textit{metric projection} on $A$, $\Pi_{A}:\mathbb{R}^{d}\rightarrow A$, is the set function 
\[
\Pi_{A}(q):=\{p\in A:d_{A}(q)=\| p-q\| \},
\]
where $d_{A}(q):=\inf_{p\in A}\| p-q\| $. We set
\[
\mathrm{Unp}A:=\{q\in\mathbb{R}^{d}:\exists!p\in A\text{ s.t. }d_{A}(q)=\| p-q\| \}.
\]
The set $\mathrm{Unp}A$ is measurable and such that $\mathrm{Leb}(\mathbb{R}^{d}\backslash\mathrm{Unp}A)=0$. Under the previous notation, the \textit{reduced normal bundle} and the \textit{reach function} of $A$ are given, respectively, by 
\[
N(A)=\left\{ \left(\Pi_{A}(q),(q-\Pi_{A}(q))/\| q-\Pi_{A}(q)\| \right):q\in\mathrm{Unp}(A)\backslash A\right\} \subseteq\partial A\times\mathbb{S}^{d-1},
\]
and, $\delta_{A}(q,u):=0$ for $(q,u)\in N(A)^{c}$, while for $(q,u)\in N(A)$, 
\begin{equation}
	\delta_{A}(q,u):=\inf\{t\geq0:q+tu\in\mathrm{Unp}(A)^{c}\}.\label{deltaA}
\end{equation}
Following \cite{kiderlenRataj06}, we will say that a closed set $A\subseteq\mathbb{R}^{d}$ is \textit{gentle} if:
\begin{enumerate}
	\item For all bounded $B\in\mathcal{B}(\mathbb{R}^{d}),$ $\mathcal{H}^{d-1}(N(\partial A)\cap(B\times\mathbb{S}^{d-1}))<\infty$.
	\item For $\mathcal{H}^{d-1}$-almost all $x\in\partial A$, there are non-degenerate balls $B_{i}\subseteq A$ and $B_{o}\subseteq A^{*}$ containing $x$.
\end{enumerate}
Thus, if $A\subseteq\mathbb{R}^{d}$ is a gentle set, then:
\begin{enumerate}
	\item For $\mathcal{H}^{d-1}$-almost all $x\in\partial A$, there is $n=n_{A}(x)\in\mathbb{S}^{d-1}$ such that $(x,n)\in N(A)$ and $(x,-n)\in N(A^{*})$. Furthermore, the mapping $x\mapsto(x,n_{A}(x))$ is measurable.
	\item It holds that $\mathrm{Leb}(\partial A)=0$. If in addition $A$ is compact, we also have that $\mathcal{H}^{d-1}(\partial A)<\infty$.
	\item Any translation of $A$ is gentle since for all $p\in\mathbb{R}^{d}$,
	$d_{A+p}(q)=d_{A}(q-p)$ and
	\[
	\Pi_{A+p}(q)=\Pi_{A}(q-p)+p.
	\]
\end{enumerate}
For $r\geq0$, the $r$-\textit{parallel set} of $A$ is defined as
\begin{equation}
	A_{\oplus r}:=\{q\in\mathbb{R}^{d}:d_{A}(q)\leq r\}.\label{defparallelset}
\end{equation}
For a more detailed exposition of the geometrical terms introduced above, see \cite{Federer96} and \cite{HugGuntWel04}. 

\section{Limit Theorems for Energy Fluxes}
Through this section we fix $m,d\in\mathbb{N}$, with $d\geq2$, $p_{0}\in\mathbb{R}^{d}$, and a bounded Lipschitz domain $\mathfrak{D}\subseteq\mathbb{R}^{d}$. We further assume that $M$, the boundary of $\mathfrak{D}$, is a $(d-1)$-dimensional compact manifold. For the rest of this paper $\partial_{k}f$ will represent the partial derivative of a function with respect to its $k$th variable.

The \textit{energy flux} of a field $X$ through the region $\mathfrak{R}=r\mathfrak{D}+p_{0}$ is defined as
\begin{equation}
	\mathscr{E}_{r}\equiv \mathscr{E}_{r}(p_0)  =\int_{rM+p_{0}}\phi(X(y))\cdot u_{rM+p_{0}}(y)\mathcal{H}^{d-1}(\mathrm{d}y), \,\,\, r>0,\label{energifluxdef}
\end{equation}
where $u_{rM+p_{0}}$ denotes the unit outward vector of $rM+p_{0}$. The scalar quantity $\mathscr{E}_{r}$ represents the flux (integral) of the vector field $\phi(X(\cdot))$ across $\mathfrak{R}$. When $\phi(X(\cdot))$ is a vector field associated with a physical quantity, $\mathscr{E}_{r}$ provides a measure of the physical element passing through the boundary of $\mathfrak{R}$. For instance, when $X$ is the velocity vector field of a fluid and $\phi(x)=\| x\| ^{2}x$, the scalar quantity $\mathscr{E}_{r}$ measures the \textit{kinetic energy} flow rate over $\mathfrak{R}$. In such a case, $\mathscr{E}_{r}$ is referred to as the \textit{kinetic energy flux}. Another example within the previous framework is when $\phi(x)=x$. In this situation, $\mathscr{E}_{r}$ quantifies the amount of fluid passing through $\mathfrak{R}$.

If we assume that $0\in\mathfrak{D}$, then $p_{0}\in \mathfrak{R}$. Thus, the quantity $\mathscr{E}_{r}(p_0)/\lvert \mathfrak{R}\rvert$ converges to the divergence of the random field $\phi(X(\cdot))$ at $p_0$ as $r\rightarrow 0$. Therefore, when $\mathscr{E}_{r}$ represents the kinetic energy flux of a fluid, the normalized integral $\mathscr{E}_{r}(p_0)/\lvert \mathfrak{R}\rvert$ converges to the divergence of the kinetic energy at $p_0$, as $\lvert \mathfrak{R}\rvert\rightarrow 0$. For turbulent fluids, this quantity represents a proxy of the energy dissipation at $p_0$.

Note that by the Divergence Theorem \sauri{and the change of variables $z=p_0+ry$,}
\[
\frac{1}{r^{d-1}}\mathscr{E}_{r}=\int_{M}[\phi(X(p_{0}+ry))-\phi(X(p_{0}))]\cdot u_M(y)\mathcal{H}^{d-1}(\mathrm{d}y).
\]
This relation illustrates that $\mathscr{E}_{r}$ can be interpreted as the ``average" (on $M$) of the increments of $\phi(X(\cdot))$ projected onto the direction of  the outward vector of $M$. In consequence, the analysis of the local behaviour of energy fluxes reduces to study the asymptotic behaviour (as $r\downarrow0$) of the functional
\begin{equation}
	Z^{\phi,r}(t,f):=\int_{M}[\phi(X(p_{0}+rty))-\phi(X(p_{0}))]\cdot f(y)\mathcal{H}^{d-1}(\mathrm{d}y),\,\,t\geq0,\label{eq:mainfunctional}
\end{equation}
where $f$ is a measurable function. \sauri{In this paper, we concentrate on the case when $f\in L^{2}(\mathcal{H}^{d-1}\downharpoonright_{M})$. Furthermore, for the rest of this work we will focus on the situation in which  $X$ is the ID field given by}
\begin{equation}
	X(p):=\int_{A+p}F(p,q)L(\mathrm{d}q),\,\,\, p\in\mathbb{R}^{d}.\label{mainfield}
\end{equation}
 \sauri{We will always assume that $L$ is an $\mathbb{R}^{m}$-valued
homogeneous L\'evy basis with characteristic triplet $\left(\gamma,\Sigma,\nu\right)$,
$F:\mathbb{R}^{d}\times\mathbb{R}^{d}\rightarrow\mathbb{R}^{d\times m}$ is of class $C^1$, and  that $A\subseteq\mathbb{R}^{d}$ is a compact set}. Note that (\ref{mainfield}) means that $i$th element of $X(p)$
follows the dynamics
\begin{equation}
	X^{(i)}(p)=\sum_{j=1}^{m}\int_{A+p}F^{(i,j)}(p,q)L^{(j)}(\mathrm{d}q),\,\,\,i=1,\ldots,d.\label{eq:mainfieldentrybyentry}
\end{equation}
Since each $L^{(i)}$ is a homogeneous L\'evy basis, and $F$ is continuous, the integrals in (\ref{eq:mainfieldentrybyentry}) are well defined in the sense of \cite{RajputRosinski89}. 

\subsection{Main Results}\label{mainresults_subsec}

In this part, we present our main findings on the functionals introduced above. We start by verifying that $Z^{\phi,r}$ is well-defined for a large class of test functions. Recall that a function $\phi:\mathbb{R}^{d}\rightarrow\mathbb{R}^{d}$ is said to be of \textit{polynomial growth} of order $\beta\geq0$ if there is some $C>0$, such that 
\[
\| \phi(x)\| \leq C(1+\| x\| ^{\beta}), \, \, \forall \, x\in\mathbb{R}^{d}.
\]

\begin{proposition}\label{welldefZ}Let $X$ be defined as in (\ref{mainfield}).
	If $\phi:\mathbb{R}^{d}\rightarrow\mathbb{R}^{d}$ is measurable and of polynomial growth of order $\beta\geq0$, then  for all $t\geq0$ and for every $f\in\mathcal{L}^{2}(\mathcal{H}^{d-1}\downharpoonright_{M})$
	\[  
	\mathbb{P}(  \lvert Z^{\phi,r}(t,f) \lvert<\infty)=1. \]
\end{proposition}

If $X$ models the velocity vector field of an \textit{incompressible} fluid, then for $\phi(x)=x$, necessarily we must have that 
\[ 
r^{-1} Z^{\phi,r}(1,u_M)=\frac{1}{r^{d}}\mathscr{E}_{r}\rightarrow0, 
\]
as $r\downarrow0$. Within our framework, this will only be the case on very specific situations. In fact, 
as pointed out in \cite{Sauri20}, the asymptotic behaviour of $Z^{\phi,r}$
strongly depends on whether $L$ is of finite variation or not. Our study in
the latter case is performed under the following assumption \sauri{(recall that we are always assuming that $L$ is homogeneous)}. 
\begin{assumption}[${\bf{A}}_\alpha$]\label{stableattractassump}
	For a given $1<\alpha\leq2$, the characteristic triplet of $L$,
	$\left(\gamma,\Sigma,\nu\right)$, satisfies the following: 
	\begin{enumerate}
		\item If $\alpha=2$, $\Sigma\neq0$. 
		\item For $1<\alpha<2$, $\Sigma=0$ and $\nu$ admits the polar decomposition in (\ref{eq:polardecomp}). Furthermore, there is a non-zero $\lambda$-integrable function $K$ such that as $s\downarrow0$ 
		\begin{equation}
			s^{\alpha}\rho_{u}(s,\infty)\rightarrow K(u),\,\,\lambda-a.a,\label{eq:Assumptionalphastablepart1}
		\end{equation}
		and
		\begin{equation}
			\sup_{u\in\mathbb{S}^{m-1}}\sup_{0\leq s\leq T}s^{\alpha}\rho_{u}(s,\infty)<\infty,\,\,\forall \, T>0 .\label{eq:eq:Assumptionalphastablepart2}
		\end{equation}	
	\end{enumerate}
\end{assumption}
\begin{remark} 
	In the $1$-dimensional case, i.e. when $m=1$, it is well known that \eqref{eq:Assumptionalphastablepart1} implies that the distribution of the ``small jumps'' of $L$ belongs to the domain of attraction of an $\alpha$-stable distribution. Not surprisingly, the same result holds in the multivariate context  under \hyperref[stableattractassump]{Assumption ${\bf{A}}_\alpha$}, see Lemma \ref{fddapproximbystable} below. Finally, we would like to emphasize that \eqref{eq:eq:Assumptionalphastablepart2} is fulfilled if \eqref{eq:Assumptionalphastablepart1} holds and either the support of $\lambda$ is finite or $\rho_{u}$ does not depend on $u$. Examples of infinitely divisible distributions on the real line satisfying \hyperref[stableattractassump]{Assumption ${\bf{A}}_\alpha$} are discussed in \cite{Sauri20}, \cite{IvanovJ18}, and references therein.
\end{remark} 
In view of Proposition \ref{welldefZ}, we will also restrict to test functions $\phi$ of polynomial growth. Thus, for $N\in\mathbb{N}$ and $\beta\geq N$, $C_{\beta}^{N}$ will denote the family of functions
$\phi:\mathbb{R}^{d}\rightarrow\mathbb{R}^{d}$ of class $C^{N}$ such that 
\[
\| \mathscr{D}^{j}\phi(x)\| \leq C(1+\| x\| ^{\beta-j}),\,\, j=0,1,2,\ldots,N,
\]
where $\mathscr{D}^{j}\phi$ denotes the vector containing all the partial derivatives of $\phi$ of order $j$. A key example is $\phi(x)=\| x\| ^{2}x$ (the test function associated to the kinetic energy) which belongs to $C_{3}^{2}$. 

Next, we introduce some auxiliary random fields that will be used for the representation of the limit of $Z^{\phi,r}$. Recall that the support function of a compact set $M$ is defined as 
\[
h_{M}(q):=\sup\{q\cdot p:p\in M\},\,\, q\in\mathbb{R}^{d}.
\]
We associate to a gentle compact set (see Section \ref{PreliminariesJCHausdorff}) $A\subseteq\mathbb{R}^{d}$ the following $\sigma$-finite measures on $\mathcal{B}(\mathbb{R}^{+}\times\mathbb{R}^{d}\times\mathbb{S}^{d-1})$: 
\[
\mu^{\pm}_{M,A}(B):=\int_{0}^{\infty}\int_{\partial A}1_{B}(s,x+p_0,\pm n_{A}(x))h_{M}(\pm n_{A}(x))^{+}\mathcal{H}^{d-1}(\mathrm{d}x)\mathrm{d}s,
\]
where $x^{+}=\max\{0,x\}$. For every $(s,x,n)\in\mathbb{R}^{+}\times N(A)$, $f\in L^{2}(\mathcal{H}^{d-1}\downharpoonright_{M})$, and $t\geq0$,
put
\begin{equation}
	G(t,f,s,x,n):=\int_{M}f(y)^{\prime}\mathbf{1}_{[h_{M}(n)s,+\infty)}(ty\cdot n)\mathcal{H}^{d-1}(\mathrm{d}y)F(p_{0},x),\label{kernellimitYalpha}
\end{equation}
where $F$ is the kernel function representing $X$ in (\ref{mainfield}). Now, for a given $\mathbb{R}^{m}$-valued homogeneous Lévy basis satisfying \hyperref[stableattractassump]{Assumption ${\bf{A}}_\alpha$}, we construct (on an extension of $\left(\Omega,\mathcal{F},\mathbb{P}\right)$) two $\mathbb{R}^{m}$-valued independent separable L\'evy bases $\Lambda_{\alpha}^{+}$ and $\Lambda_{\alpha}^{-}$ fulfilling the following: They are strictly $\alpha$-stable, independent of $\mathcal{F}$, and their control measures are $\mu^{+}_{M,A}$ and $\mu^{-}_{M,A}$, respectively. Additionally, their seed satisfies that:
\begin{enumerate}
	\item If \textbf{$\alpha=2$}, it has covariance $\Sigma$.
	\item If $1<\alpha<2$, its spectral measure is $\bar{\lambda}(\mathrm{d}u)=\alpha K(u)\lambda(\mathrm{d}u)$.
\end{enumerate}
Finally, we let
\begin{equation}
	Y^{\alpha}(t,f):=\int_{(0,t]\times N(A)}G(t,f,s,x,n)\cdot\left[\Lambda_{\alpha}^{+}(\mathrm{d}s\mathrm{d}(x,n))-\Lambda_{\alpha}^{-}(\mathrm{d}s\mathrm{d}(x,n))\right].\label{eq:limitingfields}
\end{equation}
Under the preceding notation, we have:
\begin{theorem}\label{mainthmStable} Let \hyperref[stableattractassump]{Assumption ${\bf{A}}_\alpha$} \sauri{hold for some $1<\alpha\leq2$ and consider $X$ as in (\ref{mainfield}). Suppose in addition that $A$ is a compact gentle set}. Then, for all $\phi\in C_{\beta}^{2}$, as $r\downarrow0$,
	\[
	r^{-1/\alpha}Z^{\phi,r}(t,f)\overset{\mathcal{F}\text{-}fd}{\longrightarrow}\sum_{i,j=1}^{d}D\phi(X(p_{0}))^{(i,j)}Y^{\alpha}(t,\mathbf{e}_{j}\otimes\mathbf{e}_{i}f).
	\]
	Here $D\phi$ denotes the Jacobian of $\phi$, $\mathbf{e}_{j}$ is the $j$th element of the canonical basis of $\mathbb{R}^{d}$, \sauri{and $\mathbf{e}_{j}\otimes\mathbf{e}_{i}$ represents the autoproduct between these two vectors.}
\end{theorem}

The finite variation case substantially differs from the preceding framework. More precisely:

\begin{theorem}\label{mainthmFV}\sauri{Let $X$ be defined as in (\ref{mainfield}) and assume} that $\Sigma=0$ and $\int_{\mathbb{R}^{m}}(1\land\| x\| )\nu(\mathrm{d}x)<+\infty.$ Suppose in addition that $A$ is a compact gentle set. Then, for all $\phi\in C_{\beta}^{2}$, $t\geq0$, and $f\in L^{2}(\mathcal{H}^{d-1}\downharpoonright_{M})$, as $r\downarrow0$,
	\[
	\frac{1}{r}Z^{\phi,r}(t,f)\overset{\mathbb{P}}{\rightarrow}t\sum_{i,j=1}^{d}D\phi(X(p_{0}))^{(i,j)}\mathcal{D}_{X}^{(i,j)}(f,p_{0}),
	\]
	where 
	\[ 
	\mathcal{D}_{X}^{(i,j)}(f,p_{0}):=\int_{M}f(y)^{\prime}\mathbf{e}_{i}\otimes\mathbf{e}_{j}DX(p_{0})y\mathcal{H}^{d-1}(\mathrm{d}y),
	\]
	in which  for $i,k=1,\ldots,d,$ and $\gamma_{0}=\gamma-\int_{\|x\|\leq 1}x\nu(\mathrm{d}x)$, we have let
	\begin{equation}
		DX(p)^{(i,k)}=\sum_{j=1}^{m}\int_{A+p} \partial_{k+d}F^{(i,j)}(p,q) \gamma_0^{(i)} \mathrm{d}q + \int_{A+p}\partial_{k}F^{(i,j)}(p,q)L^{(j)}(\mathrm{d}q).\label{eq:defDX}
	\end{equation}
\end{theorem}

\begin{remark} \label{remarkgenMThms}
	The following remarks are in order:
	\begin{enumerate}
		\item As mentioned above, $Z^{\phi,r}(t,f)$ can be seen as the average of the increments $\phi(X(p_0+rty))-\phi(X(p_0))$ over $M$. Therefore, in the terminology of \cite{Falconer02}, the limits appearing in Theorems \ref{mainthmStable} and \ref{mainthmFV} can be seen as the average over $M$ of all the tangent fields around $p_0$ of the field $\phi(X(\cdot))$. In fact, our techniques show that if the assumptions of  Theorem \ref{mainthmStable}  hold, then the sequence 
		\[
		r^{-1/\alpha}[\phi(X(p_0+rty))-\phi(X(p_0))],
		\]
		converges stably in distribution towards 
		\[
		D\phi(X(p_{0}))\int_{(0,t]\times N(A)}F(p_0,x)\left[\Lambda_{\alpha}^{+}(\mathrm{d}s\mathrm{d}(x,n))-\Lambda_{\alpha}^{-}(\mathrm{d}s\mathrm{d}(x,n))\right],
		\]
		where $\Lambda_{\alpha}^{\pm}$ as above but we replace $M$ by $\{y\}$. A similar result holds under the set-up of Theorem \ref{mainthmFV}.
		\item Note that $Y^{\alpha}$ is degenerated when $F(p_{0},\cdot+p_0)$ vanishes in $\partial A$, e.g. when $\partial A=\mathbb{S}^{d-1}$ and $F(p,q)=(1-\|p-q\|^2)G(p,q)$, for some vector-valued function $G$.  In fact, by looking at the proof of Theorem \ref{mainthmStable}, in this situation and as long as $A$ is gentle, the conclusion of Theorem \ref{mainthmFV} remains valid if we replace $\gamma_{0}$ by $\gamma$ in \eqref{eq:defDX}. This result is valid independently of whether \hyperref[stableattractassump]{Assumption ${\bf{A}}_\alpha$} is satisfied or not.
		\item There are other special situations in which Theorem \ref{mainthmFV} can be extended (irrespectively of the behaviour of  $F(p_{0},\cdot)$ in $\partial A$) in the infinite variation case. For instance, if $\phi(x)=x$ and $L$ is strictly 1-stable with spectral measure $\lambda$ and drift $\gamma$, our methods show that 
		\[
		\frac{1}{r}Z^{\phi,r}(t,f)\overset{\mathcal{F}\text{-}fd}{\longrightarrow}\sum_{i}Y^{1}(t,f^{i}\mathbf{e}_{i})+t\mathcal{D}_{X}^{(i,i)}(p_{0}),
		\]
		where $Y^{1}$ is defined in the same way as $Y^\alpha$ but $\Lambda_{\alpha}^{+}$ and $\Lambda_{\alpha}^{-}$ are replaced by non-trivial strictly 1-stable  L\'evy bases with spectral measure $\lambda$ and drift $\gamma$.
	\end{enumerate}
\end{remark}

\subsection{Processes induced by energy fluxes}

In this subsection, we study some probabilistic properties of the class of processes induced by the limit of energy fluxes of the form of \eqref{energifluxdef}. We start by describing the local behaviour of the energy flux
associated with $X$. In light of the relation $\mathscr{E}_{rt}=(rt)^{d-1}Z^{\phi,r}(t,u_{M})$, we deduce from Theorems \ref{mainthmStable} and \ref{mainthmFV}, and  the classical Divergence Theorem that: 

\begin{corollary}\label{resultonflux} Let $X$ be as in  (\ref{mainfield})  where $A$ is a compact gentle set. Then, for every $\phi\in C_{\beta}^{2}$, the following holds: 
	\begin{enumerate}
		\item Under \hyperref[stableattractassump]{Assumption ${\bf{A}}_\alpha$}, as $r\downarrow0$, 
		\[
		\frac{1}{r^{d-\frac{\alpha-1}{\alpha}}}\mathscr{E}_{rt}\overset{\mathcal{F}\text{-}fd}{\longrightarrow}t^{d-1}\sum_{i,j=1}^{d}D\phi(X(p_{0}))^{(i,j)}Y^{\alpha}(t,\mathbf{e}_{j}\otimes\mathbf{e}_{i}u_{M}).
		\]
		\item If $\Sigma=0$ and $\int_{\mathbb{R}^{m}}(1\land\| x\| )\nu(\mathrm{d}x)<+\infty$, then as $r\downarrow0$, 
		\[
		\frac{1}{\mathrm{Leb}(r\mathfrak{D})}\mathscr{E}_{rt}\overset{\mathbb{P}}{\rightarrow}t^{d}\sum_{i,j=1}^{d}D\phi(X(p_{0}))^{(i,j)}DX(p_{0})^{(j,i)}.
		\]
	\end{enumerate}
\end{corollary} 

It is clear that the nature of the limit processes
appearing in the previous result can be described solely by the process $(Y^{\alpha}(t,\mathbf{e}_{j} \otimes \mathbf{e}_{i}u_{M}))_{t\geq0}$. For instance, using the spectral representation \eqref{eq:limitingfields} of $Y^{\alpha}$ together with its independence from $\mathcal{F}$, we easily deduce that the limit processes in Corollary \ref{resultonflux} are self-similar of index $d-\frac{\alpha-1}{\alpha}$ and $d$, respectively.
Therefore, for the rest of this section, we focus on studying the process $(Y^{\alpha} (t,\mathbf{e}_{j} \otimes \mathbf{e}_{i}u_{M}))_{t\geq0}$. For notational convenience, from now on we will write $Y_{t}^{\alpha}$ instead
of $Y^{\alpha} (t, \mathbf{e}_{j} \otimes \mathbf{e}_{i} u_{M})$.

Our next goal is to describe the path properties of $Y^{\alpha}$ when $M$ is an affine transformation of the sphere of the form
\begin{equation}
	M=T\mathbb{S}^{d-1},\label{eq:affinetransphere}
\end{equation}
in which $T$ is an invertible $d\times d$ matrix. Note that by the self-similarity $Y^{\alpha}$ cannot be differentiable at $0$ unless it is identically zero (see Remark \ref{remarkgenMThms}). Surprisingly, however,
the paths of $Y^{\alpha}$ are typically absolutely continuous. These findings are described in the next result, in which we will use the
following notation: 
\[
\mathfrak{H}(\ell,n):=\{y\in\mathbb{R}^{d}:y\cdot n=\ell\},\,\,n\in\mathbb{S}^{d-1},\;\ell\in\mathbb{R},
\]
and 
\[
\varphi(\rho):=(1-\rho^{2})^{\frac{d-1}{2}} ,\,\,\, -1\leq\rho\leq1.
\]

\begin{theorem}\label{thmpathsYalpha}
	Let $M$ be as in (\ref{eq:affinetransphere}). Then, for all $1<\alpha<2$, the process $(Y_{t}^{\alpha})_{t\geq0}$ admits a modification that has absolutely continuous paths almost surely with derivative 
	\[
	\frac{dY_{t}^{\alpha}}{dt}=\int_{(0,t]\times N(A)}g(t,s,x,n)\cdot\left(\Lambda_{\alpha}^{+}(\mathrm{d}s\mathrm{d}(x,n))-\Lambda_{\alpha}^{-}(\mathrm{d}s\mathrm{d}(x,n))\right),
	\]
	where
	\[
	g(t,s,x,n):=\partial_{t}\varphi(s/t)(n\cdot\mathbf{e}_{i})\mathbf{e}_{j}^{\prime}F(p_{0},x)\mathcal{H}^{d-1}(T(D_{1}\cap\upsilon(n)^{\perp})),\,\,s\leq t,
	\]
	in which $\upsilon(n):=T^{\prime}n/\lVert T^{\prime}n\rVert$, $\upsilon(n)^{\perp}=\mathfrak{H}(0,\upsilon(n))$ and $D_{1}$ is the unit open disk. If $d\geq3$, then the same result holds for $\alpha=2$.
\end{theorem}
\begin{proof} 
	The proof consists in verifying that for $\mu^{\pm}_{M,A}$-a.a. $(s,x,n)\in\mathbb{R}^{+} \times \mathbb{R}^{d} \times \mathbb{S}^{d-1}$
	\begin{equation} 
		\int_{s}^{t}g(r,s,x,n) \mathrm{d}r = G(t, \mathbf{e}_{j} \otimes \mathbf {e}_{i}u_{M},s,x,n), \label{eq:acG}
	\end{equation}
	and that the stochastic Fubini theorem can be applied. From (\ref{semiclsedformkernellimitYalpha-1}) in Subsection \ref{subsec:A-kernel-representation} below, we have that for $\mu^{\pm}_{M,A}$-a.a. $(s,x,n)\in\mathbb{R}^{+}\times\mathbb{R}^{d}\times\mathbb{S}^{d-1}$ with $0<s<t$ and $h_{M}(n)>0$ it holds that
	\begin{equation}		
		G(t,\mathbf{e}_{j} \otimes \mathbf{e}_{i}u_{M},s,x,n) = (n\cdot\mathbf{e}_{i}) \mathbf{e}_{j}^{\prime} F(p_{0},x)\mathcal{H}^{d-1}(\mathfrak{D} \cap \mathfrak{H}(h_{M}(n)s/t , n)). \label{semiclsedformkernellimitYalpha}
	\end{equation}
	In what follows we fix such $(s,x,n)$. Using that $h_M(n)=\lVert T^{\prime}n\rVert$ it follows easily that
	\[
	\mathfrak{D}\cap\mathfrak{H}(h_{M}(n)s/t,n)=T(D_{1}\cap\mathfrak{H}(s/t,\upsilon(n))).
	\]
	Therefore we can parametrize $\mathfrak{D}\cap\mathfrak{H}(h_{M}(n)s/t,n)$
	as
	\[
	\psi(z)=\sqrt{1-(s/t)^{2}}Tz+\frac{s}{t}\upsilon(n),\,\,\,z\in\mathbb{S}^{d-1}\cap\upsilon(n)^{\perp}.
	\]
	Since $\mathbb{S}^{d-1}\cap\upsilon(n)^{\perp}$ is a $d-2$ dimensional sphere embedded in $\upsilon(n)^{\perp}$, we can apply the Area Formula (see, for instance, Section 3.3 in \cite{EvansGar92}) to deduce that
	\begin{equation}
		\mathcal{H}^{d-1} (\mathfrak{D}\cap\mathfrak{H}(h_{M}(n)s/t,n)) = \varphi(s/t)\mathcal{H}^{d-1}(T(D_{1}\cap\upsilon(n)^{\perp})). \label{eq:Hausdorfmmeasurehyper}
	\end{equation}
	Equation (\ref{eq:acG}) now follows easily from (\ref{semiclsedformkernellimitYalpha}) and (\ref{eq:Hausdorfmmeasurehyper}). Note that the former implies that, for all $t>0$, almost surely 
	\begin{equation}
		Y_{t}^{\alpha}=\int_{(0,t]\times N(A)}\int_{s}^{t}g(r,s,x,n)\mathrm{d}r\cdot\left[\Lambda_{\alpha}^{+}(\mathrm{d}s\mathrm{d}(x,n))-\Lambda_{\alpha}^{-}(\mathrm{d}s\mathrm{d}(x,n))\right].\label{eq:repY}
	\end{equation}
	Therefore, in order to finish the proof, we need to verify that the stochastic Fubini theorem can be applied for all $1<\alpha<2$, and for $\alpha=2$ if $d\geq3$.  Since each entry of $\Lambda_{\alpha}^{\pm}(\mathrm{d}s\mathrm{d}(x,n))$ are separable 1-dimensional strictly $\alpha$-stable L\'evy basis with control measure $\mu^{\pm}_{M,A}$, according to \cite{BasseBN11} (c.f. Lemma 3 in \cite{Sauri20}) we can swap the order of integration in
	(\ref{eq:repY}) whenever 
	\begin{equation}\label{Fubinibvres}
		\int_{0}^{t} \left(\int_{(0,r] \times N(A)}\lvert \lvert g(r,s,x,n) \lvert \rvert^{\alpha}\mu^{\pm}_{M,A}(\mathrm{d}s \mathrm{d}(x,n)) \right)^{1/\alpha} \mathrm{d}r<\infty.
	\end{equation} 
	This is easily obtained by noting that the inner integral equals to 
	\[  
	C\times r^{1-\alpha}\int_{0}^{1}\lvert\varphi^{\prime}(y)y\lvert^{\alpha}\mathrm{d}y,
	\]
	where, due to the continuity of $F$ and the compactness of the sphere, the constant	 $C$ equals
	\[  
	\int_{\partial A}\lvert\lvert(n_{A}(x)\cdot\mathbf{e}_{i})\mathbf{e}_{j}^{\prime}F(p_{0},x)\mathcal{H}^{d-1}(T(D_{1}\cap\upsilon(n_{A}(x))^{\perp}))\lvert\lvert^{\alpha}\lVert T^{\prime}n_{A}(x)\rVert\mathcal{H}^{d-1}(\mathrm{d}x)<\infty.
	\] Thus, \eqref{Fubinibvres} holds if and only if either $1<\alpha<2$ and $d\geq2$ or $\alpha=2$ and $d\geq 3$.
\end{proof}

\section{Proofs}

\sauri{Recall that we are assuming that $L$ is a homogeneous L\'evy basis. For the rest of this part, we will denote its characteristic triplet as $(\gamma,\Sigma,\nu)$ and its characteristic exponent as $\psi$}. The non-random positive constants will be denoted by the generic symbol $C>0$, and
they may change from line to line. If $A,B\subseteq\mathbb{R}^{d}$, we set
\[
A\oplus B=\{x+y:x\in A,y\in B\},
\]
and 
\[
A\ominus B=\{x\in\mathbb{R}^{d}:x-B\subseteq A\}.
\]
Let $D_{1}$ be the unit disk in $\mathbb{R}^{m}$. We will assume without loss of generality (w.l.o.g. from now on) that $M\subseteq D_{1}$. The following fact, which is a straightforward extension of Proposition 2.6 in \cite{RajputRosinski89}, will be constantly used in our proofs: If $f: \mathbb{R}^{d} \rightarrow \mathbb{R}^{N\times m}$ is integrable w.r.t. $L$, then the characteristic exponent of the infinitely divisible
random vector $\xi = \int_{\mathbb{R}^{d}} f(q) L(\mathrm{d}q) $ is given by 
\begin{equation}
	\mathcal{C}(z\ddagger\xi):=\log\mathbb{E}(\exp(\mathbf{i}z\cdot\xi))=\int_{\mathbb{R}^{d}}\psi(z^{\prime}f(q))\mathrm{d}q,\,\,\,z\in\mathbb{R}^{N}.\label{eq:cumulantfunction}
\end{equation}

Now, thanks to the L\'evy-It\^o decomposition for L\'evy basis (see \cite{Ped03}, c.f. \cite{Rosinski16}), we may and do assume that the random field defined in (\ref{mainfield}) admits the representation
\begin{equation}
	\begin{aligned}X(p)                 =&\int_{A+p}F(p,q)\gamma\mathrm{d}q+\int_{A+p}F(p,q)W(\mathrm{d}q)\\
		&+\int_{A+p}F(p,q)J_{S}(\mathrm{d}q)+\int_{A+p}F(p,q)J_{B}(\mathrm{d}q)\\
		=:& X^{(1)}(p)+X^{(2)}(p)+X^{(3)}(p)+X^{(4)}(p),
	\end{aligned}
	\label{eq:decompX}
\end{equation}
where $\gamma\in\mathbb{R}^{m}$, $W$, $J_{S}$ and $J_{B}$ are independent $\mathbb{R}^{m}$-valued homogeneous L\'evy basis with characteristic triples $(0,\Sigma,0)$, $(0,0,\nu\downharpoonright_{D_{1}})$ and $(0, 0, \nu\downharpoonright_{D_{1}^{c}}) $, respectively. Here, $\nu\downharpoonright_{D}$ denotes the restriction of $\nu$ to $D$. Moreover, $X^{(4)}$ can be written as 
\begin{equation}
	X^{(4)}(p) = \int_{A+p} \int_{\mathbb{R}^{m}} F(p,q)x \mathbf{1}_{D_{1}^{c}}(x) N(\mathrm{d}q\mathrm{d}x),
	\label{eq:eq:repX4}
\end{equation}
in which $N$ is a Poisson random measure on $\mathbb{R}^{d}\times\mathbb{R}^{m}$ independent of $(W,J_{S})$ and with intensity $\varrho=\mathrm{Leb}\otimes\nu$. Note that the latter integral is $\mathbb{P}$-a.s. well-defined in
the Lebesgue sense. If we also have that $\int_{\mathbb{R}^{m}}(1\land\| x\| )\nu(\mathrm{d}y)<\infty$, $X$ can be further decomposed as
\begin{equation}
	\begin{aligned}
		X(p) = & \int_{A+p}F(p,q)\gamma_{0}\mathrm{d}q+\int_{A+p}F(p,q)W(\mathrm{d}q)	
		+\int_{A+p}\int_{\mathbb{R}^{m}} F(p,q)xN(\mathrm{d}q\mathrm{d}x) \\
		=:& \tilde{X}^{(1)}(p)+X^{(2)}(p)+X^{(5)}(p),
	\end{aligned}
	\label{eq:decompXFV}
\end{equation}
where we have let $\gamma_{0} = \gamma-\int_{D_{1}} x \nu (\mathrm{d}x)$.

\subsection{Proof of Proposition \ref{welldefZ}}

\begin{proof}
	Fix $p_{0}\in\mathbb{R}^{d}$ and put $H(p,q):=F(p,q)\mathbf{1}_{A}(q-p)$. From (\ref{eq:decompX}), it is enough to show that
	\begin{equation}
		\| X^{(4)}(p_{0}+rty)\| \leq C\xi_{t,r},\label{eq:BJineq1}
	\end{equation}
	for some positive (finite a.s.) r.v. $\xi_{t,r}$, and that for, all $\beta\geq0$,
	\begin{equation}
		m_\beta=\mathbb{E}(\| X^{(i)}(p_{0}+rty)\| ^{\beta})\leq C,\,\,i=1,2,3;
		\label{eq:ExpX123}
	\end{equation}
	uniformly on $y\in M$. For simplicity and notational convenience, for the rest of the proof we set $p_{0}=0$. Thus, by letting $g(q,x):=\mathbf{1}_{A_{\oplus rt}}(q)\| x\| \mathbf{1}_{D_{1}^{c}}(x)$, we see that (\ref{eq:BJineq1}) holds if we set
	\[
	\xi_{t,r}:=\int_{A_{\oplus rt}}\int_{\mathbb{R}^{m}}g(q,x)N(\mathrm{d}q\mathrm{d}x),
	\]
	because $F$ is continuous and $A+rty\subseteq A\oplus rtM\subseteq A_{\oplus rt}$. The $\mathbb{P}$-a.s. finiteness of $\xi_{t,r}$ follows from the fact that $\int1\land\| g\| \mathrm{d}\varrho<\infty$ and Lemma 12.13 in \cite{Kallenberg02}. On the other hand, since $F$ is continuous and $A$ compact, it is clear that (\ref{eq:ExpX123}) holds for $i=1,2$. Now, in view that $H(p,\cdot)$ has compact support, $X^{(3)}$ has finite moments of all orders and from Corollary 1.2.6. in \cite{Turner11}, for every $\theta\geq2$,
	\[
	m_{\theta}\leq C\int_{A+p}\int_{D_{1}}(\| F(p_{0}+rty,q)x\| ^{2}\lor\| F(p_{0}+rty,q)x\| ^{\theta})\nu(\mathrm{d}x)\mathrm{d}q\leq C,
	\]
	once again by the continuity of $F$. This easily implies that (\ref{eq:ExpX123}) is also valid for $i=3$. 
\end{proof}

\subsection{Proof of Theorems \ref{mainthmStable} and \ref{mainthmFV}}

Before presenting the proof of Theorems \ref{mainthmStable} and \ref{mainthmFV}, let us make some
remarks and establish some basic results. First, since $A+p_{0}$ is gentle, w.l.o.g. we may and do assume that $p_{0}=0$. If this is not the case, replace $X$ by the ID field
\[
\tilde{X}(p) = \int_{A+p_{0}+p} F(p_{0}+p,q) L(\mathrm{d}q).
\]
For $t\in\mathbb{R}$ and $f\in L^{2}(\mathcal{H}^{d-1}\downharpoonright_{M})$, define
\begin{equation}
	Y_{r}(t,f)=\int_{M} (X(try)-X(0))\cdot f(y)\mathcal{H}^{d-1}(\mathrm{d}y).
	\label{Y_rdef}
\end{equation}
According to Lemma \ref{lemmaapproxJac} below, 
\begin{equation}\label{decompositionZ_r}
Z^{\phi,r}(t,f) = \sum_{i,j}D\phi(X(p_{0}))^{(i,j)}Y_{r}(t,f^{(i)}\mathbf{e}_{j}) + \mathrm{o}_{\mathbb{P}}(1).
\end{equation}
Thus, by the properties of stable convergence, it is enough to show that Theorems \ref{mainthmStable}
and \ref{mainthmFV} hold when we replace $Z^{\phi,r}$ by $Y_{r}$, i.e. when $\phi(x)=x$. Furthermore, in view that $f\mapsto Y_{t}^{r}(f)$ is linear, we only need to verify that the stated convergence holds for
\[
(Y_{r}(t_{1},f),\ldots,Y_{r}(t_{n},f)),
\]
for arbitrary $t_{0}:=0<t_{1}\leq\cdots\leq t_{n}$ and fixed $f\in L^{2}(\mathcal{H}^{d-1} \downharpoonright_{M})$.  Next, set $\theta_{1},\ldots,\theta_{n}\in\mathbb{R}$ and $\mathbf{T}=(t_{1}\leq\cdots\leq t_{n})$,  and let
\begin{equation}
	H(p,q):=F(p,q)\mathbf{1}_{A}(q-p), \label{HfunDef}
\end{equation}
as well as
\begin{align*}
	k_{l}^{(1)}(r,q)&:=\int_{M}f(y)^{\prime}[F(rt_{l}y,q)-F(0,q)]\mathcal{H}^{d-1}(\mathrm{d}y);\\k_{l}^{(2)}(r,q)&:=\int_{M}f(y)^{\prime}[H(rt_{l}y,q)-F(0,q)]\mathcal{H}^{d-1}(\mathrm{d}y);\\k_{l}^{(3)}(r,q)&:=\int_{M}f(y)^{\prime}H(rt_{l}y,q)\mathcal{H}^{d-1}(\mathrm{d}y).
\end{align*}
\sauri{By the Stochastic Fubini Theorem (see Lemma 3 in \cite{Sauri20} and its subsequent remark)}, we have that $\mathbb{P}$-a.s. 
\begin{equation}\label{decomp_Y_r}
\sum_{l=1}^{n}\theta_{l}Y^{r}(t_{l},f)=\Psi_{r}(\mathbf{T},f)+\Lambda_{r}(\mathbf{T},f)+\Phi_{r}(\mathbf{T},f),
\end{equation}
where 
\begin{equation}
	\begin{aligned}
		\Psi_{r}(\mathbf{T},f):=&\sum_{l=1}^{n}\theta_{l}\int_{A\cap(\cap_{k=1}^{n}A\ominus rt_{k}M)}k_{l}^{(1)}(r,q)L(\mathrm{d}q);\\\Lambda_{r}(\mathbf{T},f):=&\sum_{l=1}^{n}\theta_{l}\int_{A\backslash(\cap_{k=1}^{n}A\ominus rt_{k}M)}k_{l}^{(2)}(r,q)L(\mathrm{d}q);\\\Phi_{r}(\mathbf{T},f):=&\sum_{l=1}^{n}\theta_{l}\int_{(\cup_{k=1}^{n}A\oplus rt_{k}M)\backslash A}k_{l}^{(3)}(r,q)L(\mathrm{d}q).
	\end{aligned}
	\label{eq:fddfunctionalsdef}
\end{equation}
The following result shows that, in most cases, the leading terms are $\Lambda_{r}(\mathbf{T},f)$ and $\Phi_{r}(\mathbf{T},f)$.

\begin{theorem}\label{DivergenceThmClass}
	Let $\Psi_{r}$ be as in \eqref{eq:fddfunctionalsdef}. Suppose that, $f\in L^{2}(\mathcal{H}^{d-1}\downharpoonright_{M})$, $A$ is a compact gentle set, and $F$ is of class $C^{1}$. Then, as $r\downarrow0$,
	\[
	\frac{1}{r}\Psi_{r}(\mathbf{T},f)\overset{\mathbb{P}}{\rightarrow}\sum_{l=1}^{n}\theta_{l}t_{l}\int_{M}f(y)^{\prime}\nabla X(p_{0})y\mathcal{H}^{d-1}(\mathrm{d}y),
	\]
	where $\nabla X(p)^{(i,k)}=\sum_{j=1}^{m}\int_{A+p}\partial_{k}F^{(i,j)}(p,q)L^{(j)}(\mathrm{d}q).$ 
\end{theorem}

\begin{proof}
	For $i=1,\ldots,d,j=1,\ldots,m$, let 
	\[
	R_{F}^{(i,j)}(t,q,y):=\frac{1}{r}(F^{(i,j)}(rty,q)-F^{(i,j)}(0,q))-t\sum_{k=1}^{d}\partial_{k}F^{(i,j)}(0,q)y^{(k)}.
	\]
	By the Mean-Value Theorem and the $C^{1}$ property of $F$, we have that 
	\begin{equation}
		\left\lvert\int_{M}f(y)^{(i)}R_{F}^{(i,j)}(t,q,y)\mathcal{H}^{d-1}(\mathrm{d}y)\right\lvert\leq C\int_{M}\| f(y)\| \| y\| \mathcal{H}^{d-1}(\mathrm{d}y)<\infty,\label{eq:boundsmoothpart}
	\end{equation}
	and, as $r\downarrow0$, $\int_{M}f(y)^{(i)}R_{F}^{(i,j)}(t,q,y)\mathcal{H}^{d-1}(\mathrm{d}y)\rightarrow0$,
	due to the Dominated Convergence Theorem. Applying this to (\ref{eq:cumulantfunction}) give us that, for all $t\in\mathbb{R}$,
	\[
	\frac{1}{r}\int_{A}\int_{M}f(y)^{\prime}[F(rty,q)-F(0,q)]\mathcal{H}^{d-1}(\mathrm{d}y)L(\mathrm{d}q)\overset{\mathbb{P}}{\rightarrow}t\int_{M}f(y)^{\prime}DX(0)y\mathcal{H}^{d-1}(\mathrm{d}y).
	\]
	Hence, it is left to show that, for all $t\in\mathbb{R}$,
	\[
	\int_{A\backslash(\cap_{k=0}^{n}A\ominus rt_{k}M)}\int_{M}f(y)^{\prime}\left[\frac{F(rty,q)-F(0,q)}{r}\right]\mathcal{H}^{d-1}(\mathrm{d}y)L(\mathrm{d}q)\overset{\mathbb{P}}{\rightarrow}0.
	\]
	This is easily obtained by noticing that, by virtue of (\ref{eq:cumulantfunction}), the characteristic exponent of the latter integral equals
	\[
	\int_{A\backslash(\cap_{k=0}^{n}A\ominus rt_{k}M)}\psi\left[z\int_{M}f(y)^{\prime}\left[\frac{F(rty,q)-F(0,q)}{r}\right]\mathcal{H}^{d-1}(\mathrm{d}y)\right]\mathrm{d}q,\,\,\,z\in\mathbb{R},
	\]
	and, due to (\ref{eq:boundsmoothpart}) and the continuity of $\psi$, it is bounded up to a constant by 
	\[
	\sum_{k=1}^{n}\mathrm{Leb}(A\backslash A\ominus rt_{k}M)\rightarrow0,\,\,\text{ as }r\downarrow0,
	\]
	thanks to Theorem 1 in \cite{kiderlenRataj06}.
\end{proof}
We are now ready to present a proof of our main results.
\begin{proof}[Proof of Theorem \ref{mainthmStable}:]
	First recall that from our discussion above (see \eqref{decompositionZ_r}), we only need to concentrate on the case when $\phi(x)=x$. Furthermore, from \eqref{decomp_Y_r}, the previous theorem and our assumptions
	\begin{equation}
	\frac{1}{r^{1/\alpha}}\sum_{l=1}^{n}\theta_{l}Y^{r}(t_{l},f)  =\frac{1}{r^{1/\alpha}}\Psi_{r}(\mathbf{T},f)+\frac{1}{r^{1/\alpha}}\Lambda_{r}(\mathbf{T},f)+\mathrm{o}_{\mathbb{P}}(1),
		\label{eq:thm1concl}
	\end{equation}
	for every $1<\alpha\leq2$ .  Hence, it is enough to study the limit behaviour of the functionals $\Lambda_{r}(\mathbf{T},f)$ and $\Phi_{r}(\mathbf{T},f)$ defined in \eqref{eq:fddfunctionalsdef}. More specifically, we will show that under our assumptions, it holds
	 \begin{equation}
		\frac{1}{r^{1/\alpha}}\Lambda_{r}(\mathbf{T},f)\overset{\mathcal{F}\text{-}d}{\rightarrow}\sum_{l=1}^{n}\theta_{l}Y^{\alpha,+}(t_{l},f);\quad\frac{1}{r^{1/\alpha}}\Phi_{r}(\mathbf{T},f)\overset{\mathcal{F}\text{-}d}{\rightarrow}\sum_{l=1}^{n}\theta_{l}Y^{\alpha,-}(t_{l},f),
		\label{eq:fdconbstable}
	\end{equation} 
	where (recall the definition of $\Lambda_{\alpha}^{\pm}$ introduced in Subsection \ref{mainresults_subsec})
		\[
	Y^{\alpha,\pm}(t,f):=\pm\int_{(0,t]\times N(A)}G(t,f,s,x,n)\cdot\Lambda_{\alpha}^{\pm}(\mathrm{d}s\mathrm{d}(x,n)),\quad 1<\alpha\leq2,
	\]
	The proof then will be completed by observing that:
	\begin{enumerate}
	\item $\Lambda_{r}(\mathbf{T},f)$ and $\Phi_{r}(\mathbf{T},f)$ are independent.
	\item $Y^{\alpha}=Y^{\alpha,+}-Y^{\alpha,-}$.
	\end{enumerate}	
	
	Let us now verify that \eqref{eq:fdconbstable} is satisfied if \hyperref[stableattractassump]{Assumption ${\bf{A}}_\alpha$} holds for some $1<\alpha\leq2$ and $A$ is a compact gentle set.  By arguing as in the proof of Theorem 6 in \cite{Sauri20}, it is sufficient to check that the convergence in (\ref{eq:fdconbstable}) holds only weakly. For the rest of the proof, we restrict our attention to $\Lambda_{r}(\mathbf{T},f)$ since the arguments used in this case can be easily extrapolated to $\Phi_{r}(\mathbf{T},f)$.
	
	Let $\psi_{\alpha}$ be as in (\ref{chfnctstrictlystabe}) but $\lambda$ is replaced by $\bar{\lambda}(\mathrm{d}u)=\alpha K(u)\lambda(\mathrm{d}u)$ when $1<\alpha<2$, otherwise we set	$	\psi_{\alpha}(w)=-\frac{1}{2}w^{\prime}\Sigma w $. From Lemma \ref{fddapproximbystable} and Remark \ref{remarkapproxalphastable} below, for all $z\in\mathbb{R}$, it holds that
	\begin{equation}\label{chfnctapprox_Lamba_r}
	\mathcal{C}(z\ddagger\frac{1}{r^{1/\alpha}}\Lambda_{r}(\mathbf{T},f))=\frac{1}{r}\int_{A\backslash(\cap_{k=1}^{n}A\ominus rt_{k}M)}\psi_{\alpha}(zI_{\mathbf{T},f}^{r}(q))\mathrm{d}q+\mathrm{o}(1),
	\end{equation}
where
	\begin{equation}
		I_{\mathbf{T},f}^{r}(q):=\sum_{l=1}^{n}\theta_{l}\int_{M}f(y)^{\prime}[H(rt_{l}y,q)-F(0,q)]\mathcal{H}^{d-1}(\mathrm{d}y).
		\label{eq:Intdef}
	\end{equation}
	Now set (see Section \ref{sec:Preliminaries-and-basic}) $u\equiv n_{A}(x)$, $\delta_{+}=\delta(x,u)$ and $\delta_{-}=\delta(x,-u)$. By the continuity of $F$ and $\psi_{\alpha}$, $\psi_{\alpha}(zI_{\mathbf{T},f}^{r}(\cdot))$ is locally bounded. Therefore, we can apply Proposition 4 (Steiner's formula for gentle sets) along with the arguments of Theorem 1 in \cite{kiderlenRataj06} (c.f. \cite{HugGuntWel04}) to conclude that
	\[
	\frac{1}{r}\int_{A\backslash(\cap_{k=1}^{n}A\ominus rt_{k}M)}\psi_{\alpha}(zI_{\mathbf{T},f}^{r}(q))\mathrm{d}q=\frac{1}{r}\int_{\partial A}g_{r}(x,z)\mathcal{H}^{d-1}(\mathrm{d}x)+\mathrm{o}(1),
	\]
	in which we have let 
	\[
	g_{r}(x,z):=\int_{-\delta_{-}}^{0}\psi_{\alpha}(zI_{\mathbf{T},f}^{r}(x+su))\mathbf{1}_{(\cap_{k=1}^{n}A\ominus rt_{k}M)^{c}}(x+su)\mathrm{d}s.
	\]
	Next, we focus on showing that for ${\cal H}^{d-1}$-a.a. $x\in\partial A$
	\begin{equation}
		\frac{1}{r}g_{r}(x,z)\rightarrow\int_{0}^{t_{n}}\psi_{\alpha}(-z\sum_{l=1}^{n}\theta_{l}G(t_{l},f,s,x,-u))\mathrm{d}sh_{M}(-u)^{+},\,\,r\downarrow0.
		\label{eq:convchfncrt}
	\end{equation}
	To do this, we first note that since $A$ is gentle, for ${\cal H}^{d-1}$-a.a. $x\in\partial A$, $\delta_{+},\delta_{-}>0$ and 
	\begin{equation}
		\mathring{B}_{+}:=\mathring{B}_{\delta_{+}}(x+\delta_{+}u)\subset A^{c};\,\,\,B_{-}:=B_{\delta_{-}}(x-\delta_{-}u)\subseteq A;\label{eq:ballsoscA}
	\end{equation}
	where $B_{R}(p)$ is a ball of radius $R>0$ and centre $p$. Using this and the relation $(A\ominus rtM)^{c}=A^{c}\oplus rtM$, we deduce that for every $t>0$ it holds that 
	\begin{equation}
		\mathbf{1}_{(A\ominus rtM)^{c}}(q)=
		\begin{cases}
			1 & \text{if }q\in\mathring{B}_{+}+rtp\\
			0 & \text{if }q-rtM\subseteq B_{-}
		\end{cases},
		\label{eq:indicatorAminusrtM1}
	\end{equation}
	for some $p\in M$. Fix $x\in\partial A$ satisfying (\ref{eq:ballsoscA}) and choose $p_{-}(x)\in M$ such that $-p_{-}\cdot u=h_{M}(-u)$. From (\ref{eq:indicatorAminusrtM1}), we infer that for $r$ small enough and $\delta_{+}>s>-\delta_{-}$,
	\begin{equation}
		\mathbf{1}_{(\cap_{k=1}^{n}A\ominus rt_{k}M)^{c}}(q)=
		\begin{cases}
			1 & \text{if }s>\mathrm{o}_{+}(rt_{n})-rt_{n}h_{M}(-u)\\
			0 & \text{if }s<-\max_{k} \{rt_{k}h_{M} (-u) + \mathrm{o}_{-} (rt_{k})\}
		\end{cases},
		\label{eq:indicatorAminrtM}
	\end{equation}
	where $q=x+su$ and $\mathrm{o}_{\pm}(r)=(\delta_{\pm}-\sqrt{\delta_{\pm}^{2}-r^{2}})=\mathrm{o}(r)$. Hence 
	\begin{align}
		\frac{1}{r}g_{r}(x,z)= & \int_{-t_{n}h_{M}(-u)}^{0}\psi_{\alpha}(zI_{\mathbf{T},f}^{r}(x+rsu))\mathrm{d}s\mathbf{1}_{h_{M}(-u)>0}+\mathrm{o}(1),\label{eq:intrelation1}
	\end{align}
	where we further used that $\psi_{\alpha}(zI_{\mathbf{T},f}^{r}(\cdot))$ is locally bounded and made a change of variables. Let us now compute the limit of $I_{\mathbf{T},f}^{r}(x+rsu)$. To do this, thanks to Theorem 10.10 in \cite{Mattila99}, we may and do assume that 
	\begin{equation}
		\mathcal{H}^{d-1}\left(M\cap\{y:ty\cdot u=s\}\right)=0.
		\label{measureinter-1}
	\end{equation}
	Reasoning as in (\ref{eq:indicatorAminusrtM1}) and (\ref{eq:indicatorAminrtM}), we conclude that, for all $t\geq0$ and $q_{r}(t)=x+sru-rty$, 
	\[
	\mathbf{1}_{A}(q_{r}(t))=
	\begin{cases}
		1 & \text{if }ty\cdot u>s+\mathrm{o}_{-}(tr)/r\\
		0 & \text{if }ty\cdot u<s-\mathrm{o}_{+}(tr)/r
	\end{cases}.
	\]
	Thus,
	\begin{align*}
		I_{\mathbf{T},f}^{r}(q)=- & \sum_{l=1}^{n}\theta_{l}\int_{M}f(y)^{\prime}\mathbf{1}_{t_{l}y\cdot u<s-\mathrm{o}_{+}(rt_{l})/r}\mathcal{H}^{d-1}(\mathrm{d}y)F(0,q_{r}(t_l))+\mathrm{o}(1)\\
		\rightarrow & -\sum_{l=1}^{n}\theta_{l}\int_{M}f(y)^{\prime}\mathbf{1}_{(-\infty,s]}(t_{l}y\cdot u)\mathcal{H}^{d-1}(\mathrm{d}y)F(0,x),
	\end{align*}
	where we also used (\ref{measureinter-1}) and the fact that $\frac{\mathrm{o}_{\pm}(r)}{r} \downarrow0$ as $r \downarrow0$. The convergence in (\ref{eq:convchfncrt}) follows now by applying this to (\ref{eq:intrelation1}) and a simple
	change of variables. Finally, in view that
	\[
	\lvert g_{r}(x,z)\lvert\leq C\sum_{k=1}^{n}\int_{-\delta_{+}}^{0}\mathbf{1}_{A_{\oplus rt_{k}}^{*}}(x+sn_{A}(x))\mathrm{d}s\leq Cr,
	\]
	the limits in (\ref{eq:fdconbstable}) can now be easily obtained by the Dominated Convergence Theorem, (\ref{eq:convchfncrt}), and the fact that $\mathbf{1}_{[sh_{M}(u),+\infty)}(ty\cdot u)=0$ for $s>t.$
\end{proof}

\begin{proof}[Proof of Theorem \ref{mainthmFV}:] Observe first that in this situation 
		\begin{equation}
		\begin{aligned}
			\frac{1}{r}\sum_{l=1}^{n}\theta_{l}Y^{r}(t_{l},f)= & \sum_{l=1}^{n}\theta_{l}t_{l}\int_{M}f(y)^{\prime}\nabla X(p_{0})y\mathcal{H}^{d-1}(\mathrm{d}y)+\mathrm{o}_{\mathbb{P}}(1)\\
	& +\frac{1}{r}\Psi_{r}(\mathbf{T},f)+\frac{1}{r}\Lambda_{r}(\mathbf{T},f),
	\end{aligned}
		\label{eq:classical}
	\end{equation}
 due to  Theorem \ref{DivergenceThmClass}. Note also that thanks to Lemma \ref{fddapproximbystable}, \eqref{chfnctapprox_Lamba_r} also holds if we replace $\psi_{\alpha}$ by $\tilde{\psi}(w)=	\mathbf{i}\gamma_{0}\cdot w$. Therefore, by setting 	
 \[
 Y^{1,\pm}(t,f):=\pm\int_{(0,t]\times N(A)}G(t,f,s,x,n)\cdot\gamma_{0}\mu^{\pm}_{M,A}(\mathrm{d}s\mathrm{d}(x,n)),
 \] 
 and arguing as in the proof of Theorem \ref{mainthmStable}, we deduce that 
 	\[	 \frac{1}{r}\Lambda_{r}(\mathbf{T},f)\overset{\mathbb{P}}{\rightarrow} \sum_{l=1}^{n} \theta_{l}Y^{1,-}(t_{l},f).
 	\]
 	Similarly we have that 
 		\[	 \frac{1}{r}\Psi_{r}(\mathbf{T},f)\overset{\mathbb{P}}{\rightarrow} \sum_{l=1}^{n} \theta_{l}Y^{1,+}(t_{l},f).
 	\]
 Therefore, in order to finish the proof, we only need to check that sum of $Y^{1,+}(t,f)$ and $Y^{1,-}(t,f)$ equals
 \[	     
 t\sum_{i,k,m} \int_{M}\int_{A+p_{0}}f(y)^{(i)} \partial_{k+d}F^{(i,j)} (p_{0},x) \gamma_{0}^{(i)}y^{(k)} \mathrm{d}q \mathcal{H}^{d-1}(\mathrm{d}y).
 \]
 This relation is obtained easily by applying Fubini's Theorem, using the identity	\[
 \int_{0}^{t}G(t,f,s,x,\pm u)\mathrm{d}s=\frac{t}{h_{M}(u)}\int_{M}f(y)^{\prime}(\pm y\cdot u)^{+}\mathcal{H}^{d-1}(\mathrm{d}y)F(p_{0},x),\quad h_{M}(\pm u)>0,
 \]
and by the classical Divergence Theorem. 	\end{proof}

\subsection{Two fundamental approximations }

In this subsection, we show that in the proof of Theorems \ref{mainthmStable} and \ref{mainthmFV}, it is enough to concentrate on the case when $\phi$ is the identity function and $L$ a strictly $\alpha$-stable L\'evy basis. 

\begin{lemma}\label{lemmaapproxJac}
	Let $X$ be as in (\ref{mainfield}), with $A$ a compact gentle set and $Y_r$ as in \eqref{Y_rdef}. Then, for all $\phi\in C_{\beta}^{2}$, $f\in L^{2}(\mathcal{H}^{d-1}\downharpoonright_{M}),$ and $1<\alpha\leq2$, we have that, as $r\downarrow0$, 
	\begin{equation}
		\frac{1}{r^{1/\alpha}}\left\arrowvert Z^{\phi,r}(t,f)-\sum_{i,j}D\phi(X(p_{0}))^{(i,j)}Y_{r}(t,f^{(i)}\mathbf{e}_{j})\right\arrowvert\overset{\mathbb{P}}{\rightarrow}0.
		\label{eq:approxtaylor}
	\end{equation}
	If in addition $\Sigma=0$ and $\int_{\mathbb{R}^{m}}(1\land\| x\| )\nu(\mathrm{d}x)<\infty$, then (\ref{eq:approxtaylor}) also holds for $\alpha=1$.
\end{lemma}
\begin{proof} 
	By the Mean-Value Theorem, for all $\phi\in C_{\beta}^{2}$, it holds that
	\[
	\| \phi(x)-\phi(y)-D\phi(y)(x-y)\| \leq C(1+\| y\| ^{\beta-2})(\lVert x-y\| ^{2}\lor\| x-y\| ^{\beta}),\,x,y\in\mathbb{R}^{d}.
	\]
	As a result of this, the norm of 
	\[ 	
	\phi(X(p_{0}+rty))-\phi(X(p_{0}))-D\phi(X(p_{0}))[X(p_{0}+rty)-X(p_{0})], 
	\]
	is bounded up to a random constant (that only depends on $X(p_0)$, $f$, and $M$) by 
	\[
	\| X(p_{0}+rty)-X(p_{0})\| ^{2}\lor\| X(p_{0}+rty)-X(p_{0})\| ^{\beta}.
	\]
	In consequence, it is enough to show that, for all $i=1,2,3,4$ (remember the decomposition \eqref{eq:decompX}), as $r\downarrow0$,
	\begin{equation}
		\frac{1}{r^{1/\alpha}}\int_{M}\| f(y)\| \| X^{(i)}(p_{0}+rty)-X^{(i)}(p_{0})\| ^{\beta_{0}}\mathcal{H}^{d-1}(\mathrm{d}y)\overset{\mathbb{P}}{\rightarrow}0, \,\, \beta_{0}:=2\lor\beta.
		\label{eq:Tayloapprox}
	\end{equation}
	For simplicity and notational convenience, for the rest of the proof, we set $p_{0}=0$. Now, for $i=1,\ldots,4$, we write 
	\begin{align*}
		X^{(i)}(rty)-X^{(i)}(0) & =\int_{A\cap A\ominus\{rty\}}(F(rty,q)-F(0,q))L_{i}(\mathrm{d}q)\\
		&+\int_{A\oplus\{rty\}\backslash A}F(rty,q)L_{i}(\mathrm{d}q)\\
		& +\int_{A\backslash A\ominus\{rty\}}(H(rty,q)-F(0,q))L_{i}(\mathrm{d}q)
	\end{align*}
	where $L_{i}$ is the L\'evy basis associated to $X^{(i)}$ via (\ref{eq:decompX}) and $H$ as in \eqref{HfunDef}. Using the fact that $A\oplus rtM\subseteq A_{\oplus rt}$ and $(A\ominus rtM)^{c}\subseteq A^{c}\oplus rtD_{1}$, as well as the $C^{1}$ property of $F$, we obtain that
	\[
	\| X^{(1)}(rty)-X^{(1)}(0)\| \leq C(r+\mathrm{Leb}(A_{\oplus rt}\backslash A)+\mathrm{Leb}(A\backslash A\ominus rtD_{1})).
	\]
	Similarly, by Gaussianity and Corollary 1.2.6. in \cite{Turner11}, we infer that
	\[
	\mathbb{E}(\| X^{(2)}(rty)-X^{(2)}(0)\| ^{\beta_{0}})\leq C(r^{2}+\mathrm{Leb}(A_{\oplus rt}\backslash A)+\mathrm{Leb}(A\backslash A\ominus rtD_{1})){}^{\beta_{0}/2}.
	\]
	and 
	\[
	\mathbb{E}\left(\| X^{(3)}(rty)-X^{(3)}(0)\| ^{\beta_{0}}\right)\leq C(r^{\beta_{0}}+\mathrm{Leb}(A_{\oplus rt}\backslash A)+\mathrm{Leb}(A\backslash A\ominus rtD_{1})), 
	\]
	respectively. An application of the previous estimates and Theorem 1 in \cite{kiderlenRataj06} show that (\ref{eq:Tayloapprox}) is valid for $i=1,2,3$. On the other hand, using (\ref{eq:eq:repX4}) and arguing as above, we obtain that uniformly on $y\in M$, $\mathbb{P}$-a.s.
	\begin{equation}
		\begin{aligned}
			\| X^{(4)}(rty)-X^{(4)}(0)\| ^{\beta_{0}} & \leq C\left(r\chi+\xi_{r}\right)^{\beta_{0}}
		\end{aligned}
		\label{eq:boundX4}
	\end{equation}
	where $\chi:=\int_{A}\int_{\mathbb{R}^{m}}\| x\| \mathbf{1}_{D_{1}^{c}}(x)N(\mathrm{d}q\mathrm{d}x)$, and 
	\[
	\xi_{r}:=\int_{(A_{\oplus rt}\backslash A)\cup(A\backslash A\ominus rtD_{1})}\int_{\mathbb{R}^{m}}\| x\| \mathbf{1}_{D_{1}^{c}}(x)N(\mathrm{d}q\mathrm{d}x).
	\]
	Therefore, in order to see that (\ref{eq:Tayloapprox}) is also satisfied for $i=4$, we only need to check that 
	\begin{equation}
		r^{-1/\alpha\beta_{0}}\xi_{r}\overset{\mathbb{P}}{\rightarrow}0.\label{eq:xirdef}
	\end{equation}
	The previous relation is easily obtained by noting that $r^{-1/\alpha\beta_{0}}\xi_{r}$ is infinitely divisible with characteristic exponent 
	\[
	\mathrm{Leb}((A_{\oplus rt}\backslash A)\cup(A\backslash A\ominus rtD_{1}))\int_{\| x\| >1}(e^{\mathbf{i}zr^{-1/\alpha\beta_{0}}\| x\| }-1)\nu(\mathrm{d}x),\,\,\,z\in\mathbb{R},
	\]
	which is bounded up to a constant by $r$ (due to Theorem 1 in \cite{kiderlenRataj06}).
	
	Now suppose that $\Sigma=0$ and $\int_{\mathbb{R}^{m}}(1\land\| x\| )\nu(\mathrm{d}x)<\infty$, in such a way that (\ref{eq:decompXFV}) takes the form
	\[
	X(p)=\tilde{X}^{(1)}(p)+X^{(5)}(p)=\int_{A+p}F(p,q)\gamma_{0}\mathrm{d}q+\int_{\mathbb{R}^{d}}\int_{\mathbb{R}^{m}}H(p,q)xN(\mathrm{d}q\mathrm{d}x).
	\]
	Exactly as above, we deduce that
	\[
	\frac{1}{r}\int_{M}\| f(y)\| \| \tilde{X}^{(1)}(rty)-\tilde{X}^{(1)}(0)\| ^{\beta_{0}}\mathcal{H}^{d-1}(\mathrm{d}y)\rightarrow0.
	\]
	Moreover, (\ref{eq:boundX4}) remains valid for $X^{(5)}$ if we replace $\mathbf{1}_{D_{1}^{c}}$ by $1$ in the definition of $\chi$ and $\xi_{r}$. Therefore, in order to finish the proof, we need to verify that (\ref{eq:xirdef}) holds for $\alpha=1$ under this new definition of $\xi_{r}$. To see that this is the case, first note that the characteristic exponent $r^{-1/\beta_{0}}\xi_{r}$ now equals 
	\[
	\mathrm{Leb}((A_{\oplus rt}\backslash A)\cup(A\backslash A\ominus rtD_{1}))\int_{\mathbb{R}^{m}}(e^{\mathbf{i}zr^{-1/\beta_{0}}\| x\| }-1)\nu(\mathrm{d}x).
	\]
	Invoking once again Theorem 1 in \cite{kiderlenRataj06}, we infer that the previous quantity is bounded up to a constant by 
	\[
	r\int_{\mathbb{R}^{m}}(1\land r^{-1/\beta_{0}}\| zx\| )\nu(\mathrm{d}x)\leq Cr(1\lor r^{-1/\beta_{0}}\lvert z\lvert)\rightarrow0,\,\,\text{as }r\downarrow0,
	\]
	because $\beta_{0}\geq2$. This concludes the proof.
\end{proof}

Below, $\psi_{2}(w)=-\frac{1}{2}w^{\prime}\Sigma w$ and, for $1<\alpha<2$, $\psi_{\alpha}$ is given by (\ref{chfnctstrictlystabe}), where $\lambda$
is replaced by $\bar{\lambda}(\mathrm{d}u)=\alpha K(u)\lambda(\mathrm{d}u)$.

\begin{lemma}\label{fddapproximbystable}
	Let $\psi$ be the characteristic exponent of a homogeneous Lévy basis with triplet $(\gamma,\Sigma,\nu)$.
	Then, we have the following
	\begin{enumerate}
		\item If  \hyperref[stableattractassump]{Assumption ${\bf{A}}_\alpha$} holds for some $1<\alpha\leq2$, then, as $r\downarrow0$,
		\begin{equation}
			r\psi(r^{-1/\alpha}w)\rightarrow\psi_{\alpha}(w).
			\label{eq:conchfnct}
		\end{equation}
		\item When $\Sigma=0$ and $\int_{\mathbb{R}^{m}}(1\land\| x\| )\nu(\mathrm{d}x)<\infty$, as $r\downarrow0$,
		\[	 
		r \psi(r^{-1/\alpha}w)\rightarrow\mathbf{i} \gamma_{0}\cdot w.
		\]
	\end{enumerate}
\end{lemma} 
\begin{proof}
	We will only concentrate on the case where \hyperref[stableattractassump]{Assumption ${\bf{A}}_\alpha$} is satisfied for some $1<\alpha<2$ (the other cases are well known). In this situation, we can write, for $\tau(x)=\mathbf{1}_{\| x\| \leq1}+(1/\| x\| )\mathbf{1}_{\| x\| >1}$,
	\[
	r\psi(r^{-1/\alpha}w)=\mathbf{i}\gamma_{r}\cdot w+\int_{\mathbb{R}^{m}\backslash\{0\}}(e^{\mathbf{i}w\cdot x}-1-\mathbf{i}w\cdot x\tau(x))\nu_{r}(\mathrm{d}x),
	\]
	where 
	\[\gamma_{r}=r^{1-1/\alpha}\gamma+\int_{1<\|x\|\leq r^{-1/\alpha}}\frac{x}{\|x\|}(1-\|x\|)\nu_{r}(\mathrm{d}x)+\int_{\|x\|>r^{-1/\alpha}}\frac{x}{\|x\|}\nu_{r}(\mathrm{d}x), \] 
	and $\nu_{r}(\mathrm{d}x) = r\nu(r^{1/\alpha}\mathrm{d}x)$. According to Theorem 8.7 in \cite{Sato99}, we are left to check that 
\begin{equation}
	\begin{aligned}\int_{\mathbb{R}^{m}}f(x)\nu_{r}(\mathrm{d}x) & =r\int_{\mathbb{S}^{m-1}}\int_{0}^{\infty}f(r^{-1/\alpha}su)\rho_{u}(\mathrm{d}s)\lambda(\mathrm{d}u)\\
		& \rightarrow\int_{\mathbb{S}^{m-1}}\int_{0}^{\infty}f(su)\frac{\mathrm{d}s}{s^{1+\alpha}}K(u)\lambda(\mathrm{d}u),
	\end{aligned}
	\label{eq:Satocond1}
\end{equation}
for every continuous and bounded function $f$ vanishing on a neighborhood of $0\in\mathbb{R}^{m}$, and that 
	\begin{equation}
	\gamma_{r}\rightarrow\frac{\int_{\mathbb{S}^{m-1}}uK(u)\lambda(\mathrm{d}u)}{(1-\alpha)};\,\,\lim_{\epsilon\downarrow0}\limsup_{r\downarrow0}\int_{\|x\|\leq\epsilon}(z\cdot x)^{2}\nu_{r}(\mathrm{d}x)=0.
		\label{eq:Satocond23}
	\end{equation}
	Set $\rho_{u,r}(\mathrm{d}s):=r\rho_{u}(r^{1/\alpha}\mathrm{d}s)$ and let $\zeta_{r,u}$ be a sequence of 1-dimensional ID distributions with characteristic triplet $(0,0,\rho_{u,r})$. Equation (\ref{eq:Assumptionalphastablepart1}) and Theorem 2 in \cite{IvanovJ18} imply that $\zeta_{r,u}$ converges to a  1-dimensional strictly $\alpha$-stable random variable with Lévy measure $\alpha K(u)\frac{\mathrm{d}s}{s^{1+\alpha}}\mathbf{1}_{s>0}$. Hence, for any function $g:\mathbb{R}\rightarrow\mathbb{R}$ continuous and bounded vanishing on a neighborhood of $0\in\mathbb{R}$, it holds that for $\lambda$-almost all $u\in\mathbb{S}^{m-1}$
	\begin{equation}
		\int_{0}^{\infty}g(s)\rho_{u,r}(\mathrm{d}s)\rightarrow\int_{0}^{\infty}g(s)\frac{\mathrm{d}s}{s^{1+\alpha}}K(u),
		\label{eq:convradiusLM}
	\end{equation}
	thanks to the 1-dimensional version of Theorem 8.7 in \cite{Sato99}. The convergence in (\ref{eq:Satocond1}) now follows by applying (\ref{eq:convradiusLM}) to the function $g(s)=f(su)$ along with (\ref{eq:eq:Assumptionalphastablepart2}) and the Dominated Convergence Theorem. On the other hand, from (\ref{eq:eq:Assumptionalphastablepart2}) and Tonelli's Theorem, we deduce that, for all $\epsilon>0$,
	\[
	\int_{\| x\| \leq\epsilon}(z\cdot x)^{2}\nu_{r}(\mathrm{d}x)\leq C\int_{0}^{\epsilon}\int_{\mathbb{S}^{m-1}}r\rho_{u}(r^{1/\alpha}y,+\infty)\lambda(\mathrm{d}u)y\mathrm{d}y\leq C\int_{0}^{\epsilon}y^{1-\alpha}\mathrm{d}y,
	\]
	from which the second part of  (\ref{eq:Satocond23}) follows trivially.	Similar arguments give us that
	\begin{align*}
	\gamma_{r}=&	-\int_{\mathbb{S}^{m-1}}u\left( \int_{1}^{r^{-1/\alpha}}r\rho_{u}((r^{1/\alpha}y,1])\mathrm{d}y\right) \lambda(\mathrm{d}u)+\mathrm{o}(1)\\
	&\rightarrow-\int_{\mathbb{S}^{m-1}}uK(u)\int_{1}^{\infty}\frac{\mathrm{d}y}{y^{\alpha}}\lambda(\mathrm{d}u),
	\end{align*}
	as required.\end{proof}
\begin{remark}\label{remarkapproxalphastable}
	Let $\Lambda_{r}^{\alpha}(\mathbf{T},f)$ be as in (\ref{eq:fddfunctionalsdef}) but $L$ is replaced by: 
	\begin{enumerate}
		\item A homogeneous Gaussian Lévy basis with covariance matrix $\Sigma$ if $\alpha=2$. 
		\item $L=\gamma_{0}\mathrm{Leb}$, when $\alpha=1$. 
		\item A strictly $\alpha$-stable homogeneous Gaussian Lévy basis if $1<\alpha<2$ with spectral measure $\bar{\lambda}(\mathrm{d}u) =\alpha K(u) \lambda(\mathrm{d}u) $.
	\end{enumerate}
	By arguing as in the proof of Lemma 5 in \cite{Sauri20}, we deduce from the previous result that 
	\[
	\Lambda_{r}(\mathbf{T},f)\overset{d}{=}\Lambda_{r}^{\alpha}(\mathbf{T},f)+\mathrm{o}_{\mathbb{P}}(r^{-1/\alpha}).
	\]
	A similar approximation is valid for $\Phi_{r}(\mathbf{T},f)$.
\end{remark}

\subsection{A useful identity \label{subsec:A-kernel-representation}}

Recall that $\mathfrak{D}\subseteq\mathbb{R}^{d}$ is a bounded Lipschitz domain whose boundary $M$ is a $(d-1)$-dimensional compact manifold. For every $t>0$, set
\[
g_{t}(s,x,n):=\int_{M}u_{M}(y)\mathbf{1}_{[h_{M}(n)s,+\infty)}(ty\cdot n)\mathcal{H}^{d-1}(\mathrm{d}y)\mathbf{1}_{0<s<t}.
\]
In the next result, we find a semi-explicit representation of $g_{t}$.

\begin{proposition} 
	Suppose that $M$ is of class $C^{2}$. Then, for $\mu^{\pm}_{M,A}$-a.a. $(s,x,n) \in \mathbb{R}^{+} \times \mathbb{R}^{d} \times \mathbb{S}^{d-1}$,
	\begin{equation}
		g_{t}(s,x,n)=n\mathcal{H}^{d-1}(\mathfrak{D}\cap\mathfrak{H}(h_{M}(n)s/t,n))\mathbf{1}_{0<s<t},\label{semiclsedformkernellimitYalpha-1}
	\end{equation}
	where $\mathfrak{H}(\ell,n) := \{y\in\mathbb{R}^{d}:y\cdot n=\ell\}$. 
\end{proposition}
\begin{proof} 
	Set
	\begin{align*}
		N & = \{(s,x,n): \mathcal{H}^{d-2} (M\cap\mathfrak{H}(h_{M}(n)s/t,n))=+\infty,s\geq0\}\\
		& =\{(s,x,n):\mathcal{H}^{d-2}(M\cap\mathfrak{H}(h_{M}(n)s/t,n))=+\infty,0\leq s\leq t\},
	\end{align*}
	where the second identity follows by the definition of the support function. By Tonelli's Theorem and Theorem 10.10 in \citet{Mattila99}, $N$ is $\mu^{\pm}_{M,A}$-null set. Let us now verify that (\ref{semiclsedformkernellimitYalpha-1}) is valid for every $(s,x,n)\in N^{c}$ such that $h_{M}(n)>0$ and $t>s>0$. For such a triplet $(s,x,n)$, we have that $h_{M}(n)>\ell:=h_{M}(n)s/t>0$ and
	\begin{equation}
		\mathcal{H}^{d-1}(M\cap\mathfrak{H}(\ell,n))=0.\label{nullhyperplane}
	\end{equation}
	For every $\ell>\varepsilon>0$, let $M_{\varepsilon}$ be a $(d-1)$-dimensional compact manifold of class $C^{2}$ contained in $\{y\in\mathbb{R}^{d}:\ell-\varepsilon\leq y\cdot n\leq h_{M}(n)\}$ such that 
	\[
	M_{\varepsilon}\cap\mathfrak{H}(\ell-\varepsilon,n)=\mathfrak{D}\cap\mathfrak{H}(\ell-\varepsilon,n),
	\]
	and 
	\[
	M_{\varepsilon}\cap\{y\in\mathbb{R}^{d}:\ell\leq y\cdot n\leq h_{M}(n)\}=M\cap\{y\in\mathbb{R}^{d}:\ell\leq y\cdot n\leq h_{M}(n)\}.
	\]
	By construction, the outward vector of $M_{\varepsilon}$ satisfies that 
	\[
	u_{M_{\varepsilon}}(y)=
	\begin{cases} 
		-n & \text{if }y\cdot n=\ell-\varepsilon\\
		u_{M}(y) & \text{if }\ell \leq y \cdot n\leq h_{M}(n)
	\end{cases}.
	\]
	Thus, by the Divergence Theorem and (\ref{nullhyperplane}), we deduce that 
	\begin{equation}
		g_{t}(s,x,n) = n\mathcal{H}^{d-1}(\mathfrak{D} \cap \mathfrak{H} (\ell-\varepsilon,n)) + \mathrm{o}(\varepsilon).
		\label{kernelapprox}
	\end{equation}
	In view that $\mathfrak{D}$ is compact, there is a ball $B$ with radius $\rho>h_{M}(n)>0$ such that 
	\[
	\mathfrak{D}\cap\mathfrak{H}(\ell-\varepsilon,n)\subseteq B\cap\mathfrak{H}(\ell-\varepsilon,n).
	\]
	For $\varepsilon$ small enough, $B\cap\mathfrak{H}(\ell-\varepsilon,n)$ is a $(d-1)$-dimensional ball embedded on $\mathfrak{H}(\ell-\varepsilon,n)$ with radius $\sqrt{\rho^{2}-(\ell-\varepsilon)^{2}}$. The preceding observation allows us to apply the Generalized Dominated Convergence Theorem in (\ref{kernelapprox}) to conclude that (\ref{semiclsedformkernellimitYalpha-1}) is indeed valid. 
\end{proof} 
\bibliography{bibSept17}

\end{document}